\documentclass[a4paper,11pt]{article}
\usepackage[utf8]{inputenc}
\usepackage[T1]{fontenc}

\usepackage{wrapfig}
\usepackage{amsthm,amsmath}
\usepackage{tikz}
\usetikzlibrary{graphs,graphs.standard}
\tikzgraphsset{edges={draw,semithick}, nodes={circle,draw,semithick}}
\usepackage{mathrsfs,amssymb,amsfonts} 
\usepackage{enumitem}
\usepackage{fullpage}
\usepackage{hyperref, enumerate}
\usepackage[babel]{microtype}
\usepackage[english]{babel}
\usepackage[capitalise]{cleveref}

\usepackage{thmtools}
\usepackage{mathtools, comment}
\usepackage{amssymb}
\usepackage[nomath]{lmodern}
\usepackage{graphicx}
\usepackage{pgf,tikz,tkz-graph,subcaption}
\usetikzlibrary{arrows,shapes}
\usetikzlibrary{decorations.pathreplacing}
\usepackage{tkz-berge}
\usepackage{enumitem}
\usepackage[normalem]{ulem}
\usepackage{hyperref}
\hypersetup{colorlinks = true, linkcolor = blue, citecolor = blue, urlcolor = blue}

\newcommand*{\floorfrac}[2]{\mathopen{}\left\lfloor\frac{#1}{#2}\right\rfloor\mathclose{}}

\newcommand*{\abs}[1]{\lvert #1\rvert}

\newcommand{\diam}{diam}

\allowdisplaybreaks

\usepackage[margin=1in]{geometry}
\newtheorem{defi}{Definition}

\newtheorem{thr}[defi]{Theorem}

\newtheorem{prop}[defi]{Proposition}
\newtheorem{exam}[defi]{Example}
\newtheorem{q}[defi]{Question}
\newtheorem{remark}[defi]{Remark}
\newtheorem{claim}[defi]{Claim}
\newcommand*{\myproofname}{Proof}
\newenvironment{claimproof}[1][\myproofname]{\begin{proof}[#1]}{\end{proof}}

\title{\v Solt\'es' hypergraphs}
\author{Stijn Cambie \thanks{Department of Computer Science, KU Leuven Campus Kulak-Kortrijk, 8500 Kortrijk, Belgium. Supported by a postdoctoral fellowship by the Research Foundation Flanders (FWO) with grant number 1225224N.} }

\begin{document}
\parindent=0cm
\maketitle

\begin{abstract}
 More than $30$ years ago, \v Solt\'es observed that the total distance of the graph $C_{11}$ does not change by deleting a vertex, and wondered about the existence of other such graphs, called \v Solt\'es graphs.
 We extend the definition of \v Solt\'es' graphs to \v Solt\'es' hypergraphs,
 determine all orders for which a \v Solt\'es' hypergraph exists, observe infinitely many uniform \v Solt\'es' hypergraphs, and find the \v Solt\'es' hypergraph with minimum size (spoiler: it is not $C_{11}$).
\end{abstract}

\section{Introduction}

The total distance, which is proportional with the average distance, is one of the most important distance-based graph parameters. For a graph $G$, we denote the total distance with $W(G)$. Here $W(G)=\sum_{u,v \in V(G)} d(u,v)$ is the sum of distances between any (unordered) pair of vertices in a graph $G$.
We will examine one of the most easy-to-state questions related to the total distance.

In 1991, \v Solt\'es~\cite{Soltes91} observed that 
$W(C_{11})=W(C_{11}\setminus u)$ for every vertex $u \in V(C_{11})$,
and asked whether there are other graphs~$G$ (nowadays called \v Solt\'es' graphs) satisfying $W(G)=W(G\setminus u)$ for all vertices~$u$ of~$G$. 
It has been conjectured that $C_{11}$ is the only \v Solt\'es' graph (~\cite[conj.~51]{KST23}). 

One can generalize the notion of \v Solt\'es' graphs to \v Solt\'es' hypergraphs.
A hypergraph $H=(V,E)$, with $E \subset 2^V$, is a pair consisting of a set of vertices and a set of edges, where an edge is a set that can contain any number of vertices (we will assume at least $2$, as edges of size $1$ do not matter in computing the total distance).
More precise terminology is given in Subsection~\ref{subsec:not&def}.
A \v Solt\'es' hypergraph is a connected hypergraph $H$ for which $W(H)=W(H \setminus v)$ (which needs to be finite) for every $v \in V(H).$ 
Here $W(H)$ is the sum of distances over all pairs of vertices, where the distance $d(u,v)$ is the minimum number of edges in a connected subhypergraph of $H$ containing both $u$ and $v$.
By convention, $K_1$ is not considered as a \v Solt\'es' hypergraph.
In~\cref{sec:nonunif_Solteshypergraphs}, we prove that there exists a Solt\'es' hypergraph of order $n$ if and only if $n\ge 5$. The latter trivially implies that there are infinitly many Solt\'es' hypergraphs.

In~\cref{sec:unifSoltesH}, we consider the problem for uniform \v Solt\'es' hypergraphs. Those are hypergraphs for which all edges have the same cardinality (graphs are $2$-uniform hypergraphs). 
We give two infinite examples of families of uniform \v Solt\'es' hypergraphs, 
indicating the abundance of solutions. 
By the discovery of a few examples of uniform \v Solt\'es' hypergraphs which have larger diameter (arguably, they behave like graphs in this context), the original version of \v Solt\'es' problem may expected to be hard and one can doubt that $C_{11}$ might be the only \v Solt\'es' graph. 

There still remain natural subproblems, such as determining the smallest uniform \v Solt\'es' hypergraphs.
We find two uniform \v Solt\'es' hypergraphs of order $10$ (hence $C_{11}$ is not the smallest uniform \v Solt\'es' hypergraph); the hemi-dodecahedron (see~\cref{fig:Petersen}) and the dual of $K_5^{(3)}$ (see~\cref{fig:BIBD}). We prove that the latter is the unique Solt\'es' hypergraph of minimum size.

\subsection{Notation and Terminology}\label{subsec:not&def}

A hypergraph $H=(V,E)$, where $E \subset 2^V$, consists of a set of vertices and a set of hyperedges. A hyperedge is defined as a set containing at least two vertices\footnote{Edges with a single vertex would not influence the total distance.} and called an edge when no confusion can arise. The number of vertices and edges are called the order and size respectively.
For a vertex $v \in V$, $H \setminus v$ denotes the hypergraph $(V \setminus v, E')$, where $E'$ includes all hyperedges in $E$ not containing $v$. In essence, $H \setminus v$ is obtained by removing $v$ and all hyperedges containing $v$.
The dual $H^\star=(E, V^\star)$ of a hypergraph $H=(V,E)$, is the hypergraph whose vertices are the edges of $H$, and whose hyperedges correspond with the tuples of all edges incident to a vertex of $V.$ A clique $K_n^{(r)}$ is the $r$-uniform hypergraph with a vertex set of order $n$, containing all possible hyperedges with $r$ elements.
So up to isomorphism, $K_n^{r}=\left( [n], \binom{[n]}{r} \right)$ where $[n]=\{1,2,\ldots,n\}.$

The degree of a vertex $v$, denoted by $\deg(v)$, is the number of hyperedges containing $v$. The minimum and maximum degrees, denoted as $\delta$ and $\Delta$, represent the smallest and largest degrees in the hypergraph.

A path in $H$ is represented by a sequence $(v_0, e_1, v_1, e_2, v_2, \ldots, e_r, v_r)$, where each $e_i$ is a hyperedge containing vertices $v_{i-1}$ and $v_i$ for every $i$, $v_0=u$, and $v_r=v$. No edge or vertex appears multiple times in the path, and $r$ denotes the length of the path.
The distance $d(u,v)$ between two vertices $u,v$ in a hypergraph $H$ is defined as the length of the shortest path between them. 

For a hypergraph $H$, the total distance or Wiener index, $W(H)$, is the sum of distances between any pair of vertices in $H$, i.e., $W(H)=\sum_{u,v \in V(H)} d(u,v)$. Equivalently, $W(H)$ equals the total distance $W(G')$ of the underlying graph $G'=(V,E')$, where $E'= \large\cup \{ \binom{e}{2} \mid e \in E\}$ comprises all $2$-uniform edges that are subsets of a hyperedge in $E$.

\section{Non-uniform \v Solt\'es' hypergraphs}\label{sec:nonunif_Solteshypergraphs}

Before proving the existence of non-uniform \v Solt\'es' hypergraphs of almost any order, we present a few small examples.

\begin{exam}\label{exam:5-8}
 Consider the hypergraphs in~\cref{fig:n=5-8}.
 Here edges are drawn with lines, and hyperedges with dashed ellipses.
 These satisfy 
 $W(H)=W(H \setminus v)=\binom{n}{2}$ for any $v \in V(H)$.
\end{exam}

\begin{figure}[h]
 \centering
 \begin{tikzpicture}[scale=0.75*1.05]
 \foreach \x in {0,72,...,288}{\draw[fill] (\x:1.5) circle (0.15);
\draw(\x+72:1.5) -- (\x:1.5);
}
\draw[dashed] (0,0) circle (2.25);
 \end{tikzpicture} \quad
 \begin{tikzpicture}[scale=0.75*1.05]
\foreach \x in {0.5,1.5}
{
\draw[fill] (0,\x) circle (0.15);
\draw[fill] (1,1) circle (0.15);
\draw (0,\x)--(-1,-1);
}
\foreach \x in {-0.5,-1.5}
{
\draw (0,\x)--(1,1);
\draw[fill=white] (0,\x) circle (0.15);
\draw[fill=white] (-1,-1) circle (0.15);
}
\draw[dashed] (1/3,1) circle (0.9);
\draw[dashed] (-1/3,-1) circle (0.9);
\draw[dashed] (0,0) circle (2.25);
 \end{tikzpicture} \quad
\begin{tikzpicture}[scale=0.75*1.05]
\foreach \x in {0.5,1.5}
{
\draw[fill] (0,\x) circle (0.15);
}
\draw[fill] (1.5,1)--(0,1.5);
\draw[fill] (-1.5,1)--(0,0.5);
\draw[fill] (1.5,1) circle (0.15);
\draw[fill] (-1.5,1) circle (0.15);
\foreach \x in {-0.5,-1,-1.5}
{
\draw (0,\x)--(1.5,1);
\draw (0,\x)--(-1.5,1);
\draw[fill] (0,\x) circle (0.15);
}
\draw[dashed] (0.65,1) circle (1.05);
\draw[dashed] (-0.65,1) circle (1.05);
\draw[dashed] (0,0.25) circle (2.25);

\draw[dashed] (0,-1) ellipse (0.3cm and 0.85cm);
 \end{tikzpicture} \quad
 \begin{tikzpicture}[scale=0.5625*1.05]
 \foreach \x in {0,45,...,315}{\draw[fill] (\x:2.5) circle (0.15);
\draw(\x+45:2.5) -- (\x:2.5);
\draw(\x+90:2.5) -- (\x:2.5);
}
\foreach \x in {0,90,180,270}
{
\draw[fill=white] (\x:2.5) circle (0.15);
}
\draw[dashed] (0,0) circle (3.0);
\foreach \x in {0,45}{
\draw[dashed] (\x+45:2.89) -- (\x-45:2.89)--(\x-135:2.89)--(\x+135:2.89)--cycle;
}
 \end{tikzpicture}

 \caption{Examples of non-uniform \v Solt\'es' hypergraphs for $n=5,6,7,8$}
 \label{fig:n=5-8}
\end{figure}
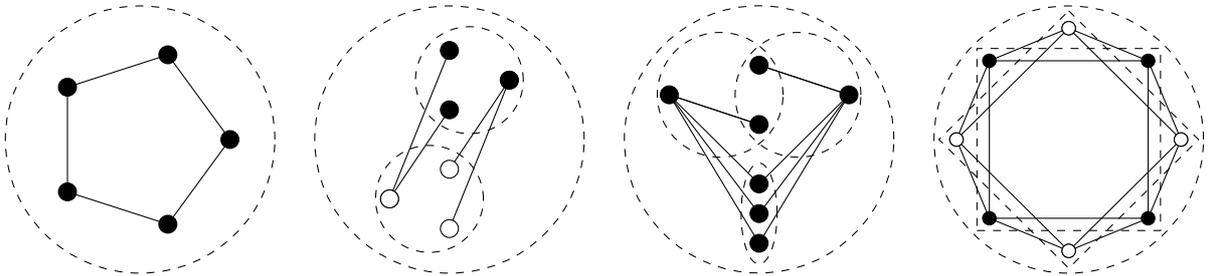

We give a few remarks for the hypergraph $H$ on $7$ and $8$ vertices. 

\textbf{Case $n=7$: } Since $H$ has three orbits, there are three computations needed to verify that $W(H \setminus v)=W(H)=\binom{7}2=21$ for every $v \in V(H).$ 
 The hypergraph is non-regular, since its minimum degree equals $\delta=4$ and the maximum degree is $\Delta=6$. 
 This imply that some conjectures~\cite[conj.~47,48]{KST23} stating that a \v Solt\'es' graph is regular or vertex-transitive, are not true for the extension towards \v Solt\'es' hypergraphs.
 Also, $H$ contains twin-vertices, which cannot happen in a \v Solt\'es' graph.

\textbf{Case $n=8$: } The hypergraph on $8$ vertices is obtained from the circulant graph on $[8]$ with difference set $\{1,2\}$ together with hyperedges containing the even, odd and all vertices. Since the hypergraph is vertex-transitive, it is easy to verify that 
 $W(H)=\binom{n}{2}=28$ and $W(H \setminus v)=W(K_4)+W(P_3)+3\cdot(2 \cdot 2+2\cdot 1)=6+4+18=28$ for any $v \in V(H)$.

Next, we generalize the folklore result on the extrema of the Wiener index for graphs of a given order. This will be used afterwards to prove that no \v Solt\'es' hypergraph can have order bounded by $4$.

\begin{prop}\label{prop:folklore}
 For a connected hypergraph $H$ on $n$ vertices, we have
 $$\binom{n}{2}=W(K_n) \le W(H) \le W(P_n)=\binom{n+1}{3}.$$
 The upper bound is strict if $H \not= P_n$, while there are many hypergraphs attaining the minimum.
\end{prop}

\begin{proof}
 The lower bound is trivial, since all distances are at least one.
 The upper bound is by induction.
 Let $S$ be a minimum set of (twin-)vertices such that $H \setminus S$ (remaining hypergraph after removing $S$ and edges incident with $S$) is connected. The existence of such a set can be concluded by starting from a vertex and adding overlapping edges one at a time. 
 Let $\abs S=s.$
 Then by induction, we have that 
 \begin{align*}
 W(H) &\le H(H \setminus S) + s \cdot \sum_{i=1}^{n-s} i\\ &= \binom{n-s+1}{3}+s \binom{n-s+1}{2}\\&\le
 \binom{n-s+1}{3}+ \sum_{i=n-s+1}^{n} \binom{i}{2}=\binom{n+1}{3}.
 \end{align*} 
 For equality to be possible, it is necessary that $s=1$ and $H \setminus S=P_{n-1}$ (by induction). From here one easily concludes that $H=P_n$, since the distances between the vertex $S$ and the vertices in $P_{n-1}$ have to be exactly $1$ up to $n-1.$
\end{proof}

We end this section with the main result, the determination of all orders for which a \v Solt\'es' hypergraph exists.

\begin{thr}\label{thr:smallestSH}
 The smallest \v Solt\'es' hypergraph has order $5$ and is unique. For every $n \ge 5$, there exists a \v Solt\'es' hypergraph of order $n.$
\end{thr}

\begin{proof}
 For $2 \le n \le 4$, by~\cref{prop:folklore}, $W(H \setminus v)\le \binom{n}{3}<\binom{n}{2} \le W(H)$ if both $H$ and $H \setminus v$ are connected (which is a necessary condition).
 For $n=5$, equality needs to be attained, so $H\setminus v \cong P_4$ for ever $v$ and $W(H)=10$, implying that $H$ is exactly equal to a $C_5$ with one hyperedge containing the $5$ vertices. 

 For $5 \le n \le 8$, examples of \v Solt\'es' hypergraphs are depicted in~\cref{fig:n=5-8}.

 Finally, we show that there are \v Solt\'es' hypergraphs of order $n$ for every $n \ge 9$, by giving an example, which actually depends on $n \pmod {12}.$

 Partition the vertex set $V$ in $\floorfrac{n}3$ sets $A_i$ of size $3$ or $4.$
 Let $E_i=V \setminus A_i$ for every $1 \le i \le \floorfrac n3$.
 We start with constructing a graph $G$ of order $n$.
 
 For a set $A_i$ of size $4$, we choose $G[A_i]$ to be a $m$-factor, where $n-1 \equiv 3-m \pmod{3}$ and every vertex from $A_i$ has $\frac{2(n-4)-m}3$ neighbours in $E_i$. Note that every two of them have at least one common neighbour since $2\frac{2(n-4)-m}3>n-4$.

 For a set $A_i$ of size $3$, if $2 \mid n-1$, we choose either 
 \begin{itemize}
 \item $G[A_i]$ to be a triangle $K_3$ and each vertex in $A_i$ has $(n-3)-\frac{n-1}{2}$ neighbours in $E_i$
 \item $G[A_i]$ to be a single edge, the two vertices of the $K_2$ in $A_i$ both having $(n-3)-\frac{n-1}{2}$ neighbours in $E_i$ and let the third vertex have degree $(n-3)-\frac{n-1}{2}+1$ and have a common neighbour with each of the other vertices.
 \end{itemize}

 For a set $A_i$ of size $3$, if $2 \mid n-2$, we choose either
 \begin{itemize}
 \item $G[A_i]$ to be an empty graph $3K_1$ and each vertex in $A_i$ has $(n-3)-\frac{n-2}{2}$ neighbours in $E_i$ (such that every two vertices of $A_i$ have a common neighbour) 
 \item $G[A_i]$ to be a $P_3$, where the two endvertices in $A_i$ have $(n-3)-\frac{n-2}{2}$ neighbours in $E_i$ (at least one in common) and the third vertex of $A_i$ has degree $(n-3)-\frac{n}{2}.$
 \end{itemize}

 Since $n>8$, there is at least one $A_i$ of size $3$ and so the choices above can be made in such a way that the sum of degrees is even.
 There are multiple ways to see that a valid construction is possible for every order $n \ge 8$. One can e.g. do so by giving explicit constructions depending on $n \pmod{12}.$
 Finally, add the hyperedges $E_i$ for every $1 \le i \le \frac n3$ and $V$ to $G$,
 to form a hypergraph $H.$
 Now one can check that $W(H)=W(H \setminus v)=\binom n2$ in all cases.
 Since $H \setminus v$ has diameter $2$, it is sufficient to check that there are $n-1$ pairs of vertices at distance $2$. If $v \in A_i,$ $E_i$ is a clique in $H \setminus v$, thus one needs to count the non-neighbours of the remaining vertices in $A_i$. The $n-1$ pairs are now obtained by construction.
 If $\abs{A_i}=4$, this is immediate by construction, since there are $3-m$ pairs with two vertices in $A_i$ at distance $2$ and $n-4+m=(n-1)-(3-m)$ nonadjacent pairs of vertices in $A_i \times E_i.$
 If $\abs{A_i}=3,$ the same conclusion holds since $2\frac{n-1}{2}, 1+2\frac{n-2}{2}, \frac n2+\frac{n-2}{2}$ are all equal to $n-1$.

 Examples of such hypergraphs are depicted in~\cref{fig:n9&10} for $n \in \{9,10\}$ (where two hyperedges for the union of two colour classes (sets $A_i$) are not depicted for clarity of the drawing).
\end{proof}

\begin{figure}[h]
 \centering
 \begin{tikzpicture}[scale=1.05*1.05]
 \foreach \x in {0,40,...,320}{\draw[fill] (\x:2.5) circle (0.15);
\draw(\x+40:2.5) -- (\x:2.5);
\draw(\x+120:2.5) -- (\x:2.5);
}
\draw[dashed] (0,0) circle (3.06);
\foreach \x in {0,120,240}
{
\draw[fill=white] (\x:2.5) circle (0.15);
\draw[fill=red] (\x+40:2.5) circle (0.15);
}
\foreach \x in {0}{
\draw[dashed] (\x+40:2.89) -- (\x-40:2.89)--(\x-80:2.89)--(\x-160:2.89)--(\x+160:2.89)--(\x+80:2.89)--cycle;
}
 \end{tikzpicture} \quad
 \begin{tikzpicture}[scale=1.05]
 \foreach \x in {2}{
 \foreach \y in {0}{
 \draw[black!20!white](-0.5+0.5*\x,0.866-0.866*\x) -- (3.5+0.5*\y,-0.866+0.866*\y);
 };
 }

 \foreach \z in {2.5,3.5}{
 \foreach \y in {0,1,2}{
 \draw[black!20!white](\z,4*0.866) -- (3.5+0.5*\y,-0.866+0.866*\y);
 };
 }

 \foreach \z in {0.5}{
 \foreach \y in {1}{
 \draw[black!60!white](\z,4*0.866) -- (3.5+0.5*\y,-0.866+0.866*\y);
 };
 }

 \foreach \z in {1.5}{
 \foreach \y in {2}{
 \draw[black!60!white](\z,4*0.866) -- (3.5+0.5*\y,-0.866+0.866*\y);
 };
 }

 \draw (0.5,-0.866)--(3.5,-0.866);
 \foreach \z in {0.5,1.5}{
 \foreach \x in {0,1,2}{
 \draw[black!20!white](\z,4*0.866) -- (-0.5+0.5*\x,0.866-0.866*\x) ;
 };
 }

 \foreach \z in {2.5}{
 \foreach \x in {0}{
 \draw[black!60!white](\z,4*0.866) -- (-0.5+0.5*\x,0.866-0.866*\x) ;
 };
 }
 
 \foreach \z in {3.5}{
 \foreach \x in {1}{
 \draw[black!60!white](\z,4*0.866) -- (-0.5+0.5*\x,0.866-0.866*\x) ;
 };
 }

 \foreach \x in {0,1,2}{
 \draw[fill=white] (-0.5+0.5*\x,0.866-0.866*\x) circle (0.15);
 }
 \foreach \y in {0,1,2}{
 \draw[fill=red] (3.5+0.5*\y,-0.866+0.866*\y) circle (0.15);
 }
 \foreach \z in {0.5,1.5,2.5,3.5}{
 \draw[fill] (\z,4*0.866) circle (0.15);
 }

 \draw[dashed] (-0.8,1.1) -- (0.2,-1.1)--(3.8,-1.1)--(4.8,1.1)--cycle;

 \draw[dashed] (2,4/3) circle (3.23);

 \end{tikzpicture}
 
 \caption{\v Solt\'es' hypergraphs $H$ for $n=9, 10$}
 \label{fig:n9&10}
\end{figure}
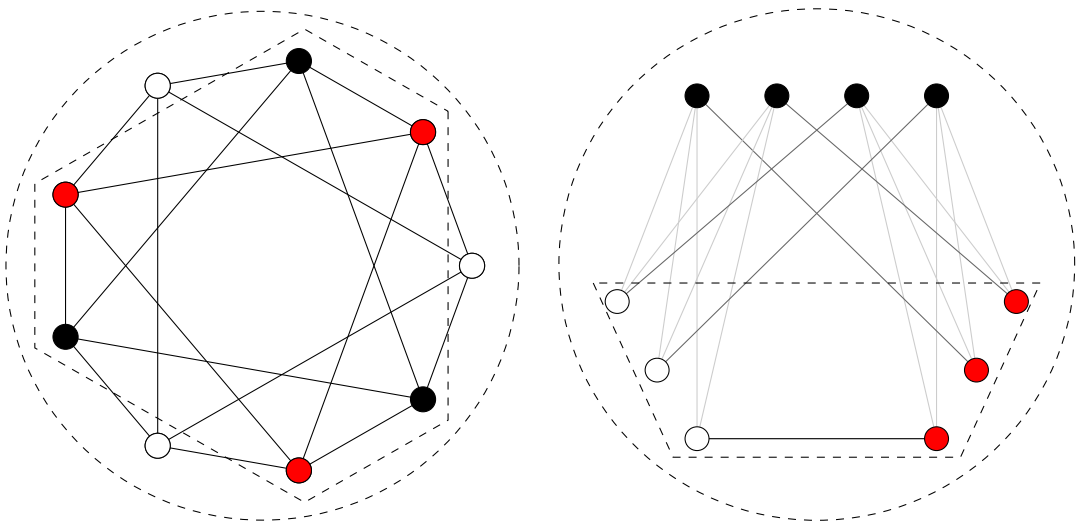

We conclude that there are numerous examples of non-uniform \v Solt\'es' hypergraphs, and additional constructions are presented in~\cref{sec:app}. 
To narrow our focus and potentially gain deeper insights, we turn our attention to a specific subclass of hypergraphs. Considering that graphs correspond to $2$-uniform hypergraphs, the most natural class to explore is that of uniform hypergraphs. Our hope is that understanding uniform \v Solt\'es' hypergraphs will provide valuable insights in the search for \v Solt\'es' graphs. This is exactly the scope of the next section.

\section{Uniform \v Solt\'es' hypergraphs}\label{sec:unifSoltesH}

In this section, we prove that there are (infinitely many) uniform \v Solt\'es' hypergraphs. We do so by constructing two infinite families; one with increasing uniformity and one with fixed uniformity.

\begin{prop}\label{prop:cyclelikeSH}
    There are infinitely many values of $k$ for which there exists a $k$-uniform \v Solt\'es' hypergraph.
\end{prop}

\begin{proof}
    Let $m \ge 4,$ $k=\binom{m}{2}$ and $n=\binom{m+1}{2}+1.$
    Let $H$ be the hypergraph with $V=[n]$ and $E$ be the edges consisting of $k$ consecutive (modulo $n$) numbers in $[n]$.
    Then $H$ is vertex-transitive and $W(H)=\binom{n}{2}$, as $\diam (H)=1$ due to $2k-1 \ge n$). For any $v \in V$,
    $W(H \setminus v) = \binom{n-1}2+\sum_{i=1}^m i =\binom n2.$
    The latter since $H \setminus v$ has diameter $2$ and it is equal to a path for which the first $m$ vertices have resp. $m,m-1 \ldots, 1$ non-neighbours in $V \setminus v.$
\end{proof}

\begin{prop}\label{prop:Steiner24nSH}
 There are infinitely many $4$-uniform \v Solt\'es' hypergraphs.
\end{prop}

\begin{proof}
 Consider a Steiner $2$-design $H=S(2,4,n)$ with block size $4$, for some $n \equiv 1,4 \pmod{12}$ and $n \ge 13$. They exist for all such $n$, and there are many of them for large $n$, see~\cite{RR10}.
 
 Every $2$ vertices share exactly one edge, thus $\diam(H)=1$. If $u,w$ are two vertices such that $\{u,v,w\}$ are part of an edge $e$ in $H$, then for a vertex $z \not \in e$, there are two disjoint edges (in $G \setminus v$ which contain $\{z,u\}$ and $\{z,w\}$.
 This implies that for every vertex $v \in V(H),$
 $$W(H\setminus v)= \binom{n-1}{2}+\frac{n-1}{3}\binom{3}{2}=\binom{n-1}{2}+(n-1)=\binom{n}2=W(H).$$
 
 Thus $H$ is indeed a \v Solt\'es' hypergraph, which has diameter $1$ and is $4$-uniform.
\end{proof}

The examples in the above families all have diameter $1$,
making the computations easier.
We remark that there are no $3$-uniform \v Solt\'es' hypergraphs with diameter $1$. In that regard, finding a $3$-uniform \v Solt\'es' hypergraph may be somewhat equally difficult as finding a \v Solt\'es' graph. 
Finding one, or excluding the existence seems a hard question.
From any cubic arc-transitive graph $G$, one can derive a $2$-regular $3$-uniform hypergraph $H$ with $V(H)=E(G)$ and $E(H)$ being exactly equal to the triangles in the line graph of $G$.
Using the census on cubic arc-transitive graphs~\cite{CD02}, see \url{https://www.math.auckland.ac.nz/~conder/AllSymmCubic-OrderUpTo10000-ByEdgeSets.txt},
by taking the dual,
we find many examples of $3$-uniform hypergraphs $H$ for which $W(H)-W(H \setminus v)$ can be either positive or negative, and there are examples for which it is exactly equal to $1$ or $-1$, or near zero compared with its order. See~\cite[doc. \text{arctrans10000\_3unif}]{C24}.

\begin{prop}\label{prop:no3unifdiam1SH}
    There are no $3$-uniform \v Solt\'es' hypergraphs with diameter $1.$
\end{prop}

\begin{proof}
    Let $H$ be a $3$-uniform \v Solt\'es' hypergraph with diameter $1.$
    We associate with every pair of vertices $(v_i,v_j)$ in $\binom{V}{2},$ all vertices $v' \in V\setminus\{v_i,v_j\}$ for which $\{v_i,v_j,v'\} \in E.$
    For every vertex $v'$, count the number $x_{v'}$ of pairs associated with $v'$ and no other vertex.
    Since there are $\binom n2$ pairs in total, and $n$ vertices, there is a vertex $v$ for which this number is bounded by $\frac{n-1}{2}.$
    Then it is easy to see that $H \setminus v$ has diameter $2$ as follows. For every $2$ vertices $u,w$ in $V \setminus v,$ which are not adjacent in $H \setminus v$ there are at most $\frac{n-3}{2}$ vertices which are not a neighbour of both $u,w$.
    So they have at least $\frac{n-3}{2} \ge 1$ (we only need to consider $n \ge 5$ by~\cref{thr:smallestSH}) neighbours in common.
    Since $\diam(H \setminus v)=2,$ we conclude that 
    $W(H \setminus v)=\binom{n-1}{2}+x_v<\binom{n}{2}$, contradicting the choice of $H$.
\end{proof}

\begin{remark}\label{rem:4unif_diam1&nonVT}
Among the $18$ examples for $ S(2, 4, 25)$, as mentioned in~\cite{RR10}, there are also non-vertex-transitive examples.

With the same idea as for~\cref{prop:no3unifdiam1SH}, one can prove that all $4$-uniform \v Solt\'es' hypergraphs for which $\diam (H)=1$ and $\diam(H \backslash v)=2\ \forall v \in v(H)$ were determined in~\cref{prop:Steiner24nSH}. 
\end{remark}

A first example of a \v Solt\'es' hypergraph with diameter larger than $1$ is obtained by duplicating the dual of $K_4$ (equivalently the Fano plane with one vertex removed), with the corresponding hyperedges put together, and where the vertex set of every copy is a hyperedge as well.
This one is vertex-transitive, but not (hyper)edge-transitive (there is exactly one pair of complementary hyperedges). It is presented in~\cref{fig:diam2_F^*} (left), with $4$ hyperedges drawn by dashed lines, and also the vertices with the same colour belonging to a hyperedge.

\begin{remark}
 The $6$-uniform hypergraph on $\{0,1,2,\ldots,11\}$ given by its hyperedges
 $$((0, 1, 2, 3, 4, 5), (0, 1, 6, 7, 8, 9), (2, 3, 6, 7, 10, 11), (4, 5, 8, 9, 10, 11), (0, 2, 4, 6, 8, 10), (1, 3, 5, 7, 9, 11))$$ is a (uniform) \v Solt\'es' hypergraph with diameter $2.$
\end{remark}

\begin{figure}[h]
 \centering
 \begin{tikzpicture}[scale=0.5]
 \foreach \x in {30,150,270}{
 \draw[fill] (\x+5:2.5) circle (0.15);
  \draw[fill=white] (\x-5:2.5) circle (0.15);
}
 \foreach \x in {90,210,330}{
 \draw[fill] (\x+5:5) circle (0.15);
  \draw[fill=white] (\x-5:5) circle (0.15);
}
 \draw[dashed] (0,0) circle (3);

\draw[dashed,rotate=30] (2.5,0) ellipse (1.6cm and 5.25cm);
\draw[dashed,rotate=150] (2.5,0) ellipse (1.6cm and 5.25cm);
\draw[dashed,rotate=270] (2.5,0) ellipse (1.6cm and 5.25cm);
 \end{tikzpicture} \quad
 \begin{tikzpicture}[scale=1.15]
 \foreach \x in {0,72,144,216,288}{
 \draw[fill] (\x:2.5) circle (0.1);
  \draw[fill=white] (\x:1.3) circle (0.1);
  \foreach \y in {-24,24}{
    \draw[fill] (\x+\y:2.08) circle (0.1);
  }
}
 \foreach \x in {0,72,144,216,288}{
\draw[dashed,rotate=\x+36] (0:1.05) ellipse (0.34cm and 2.2cm);
\draw[dotted,rotate=\x] (-1.85,-1.75)--(-2.175,-1.75)--(-2.175,1.75)--(-1.85,1.75)--cycle;
}

 \end{tikzpicture} 
 
 \caption{Uniform \v Solt\'es' hypergraphs $H$ with diameter $2$}
 \label{fig:diam2_F^*}
\end{figure}
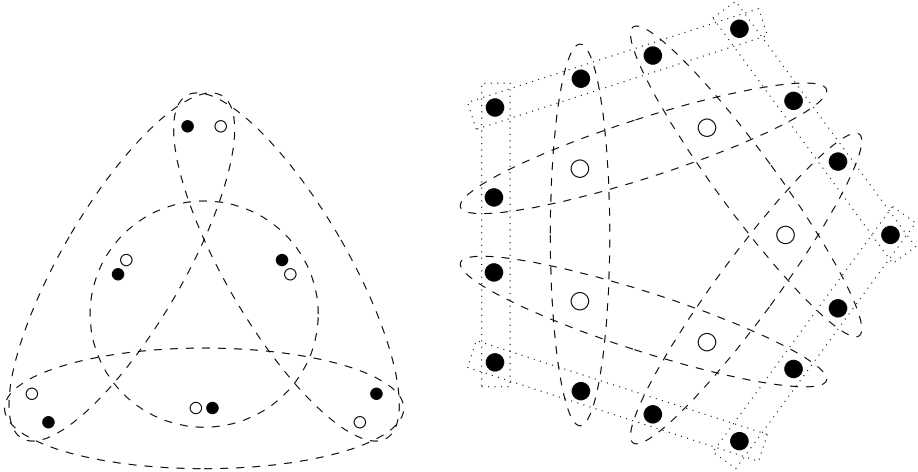

All the constructions we have given so far had diameter $1$ or $2$. The simple reason for this is that $W(H)-W(H \setminus v)$ being zero is a rare event and thinking about hypergraphs with larger diameter is hard.
Using the census of \url{https://staff.matapp.unimib.it/~spiga/census.html}~\cite{POSV13, PSV15} and a computer search, we managed to find $4$ additional \v Solt\'es' hypergraphs, having diameter $2,6,7$ and $8$.
They are $2$-regular (as is the case for $C_{11}$) and $4$-uniform, being the dual hypergraphs of $4$-regular $2$-uniform hypergraphs (graphs $45705, 51044, 51045, 51046$ from \url{https://houseofgraphs.org/graphs/}), whose graph6-strings are also listed in~\cref{sec:appb}. The smallest one (the dual of the blow-up of a $C_5$) being depicted in~\cref{fig:diam2_F^*} (right).
Note that since the initial graphs are arc-transitive, the dual hypergraph is vertex-transitive and so it is sufficient to consider $W(H)-W(H \setminus v)$ for just one vertex $v$. The full verification can be found at \cite[doc. \text{graphs\_dual\_SH}]{C24}.

Subjectively, they may be considered as surprising as the possible existence of \v Solt\'es' graphs different from $C_{11}$ (since there are more uniformities to consider, there were more chances for the existence of course) and evidence that excluding the existence of \v Solt\'es' graphs different from $C_{11}$ is likely very hard if not impossible.

The examples so far are all somewhat large (order at least $11$). Next, we present a uniform \v Solt\'es' hypergraph, which has order $10.$ This may well be the minimum possible order.

\begin{exam}\label{exam:10_5}
Consider the Petersen graph $P_{5,2}$, as e.g. depicted in~\cref{fig:Petersen} (left) and choose a $(6,2)$-cycle-cover, that is a collection of $6$ smallest cycles in $P_{5,2}$ that cover every edge twice (this is easily seen to be unique up to isomorphism).
Equivalently, choose the $6$ faces of the hemi-dodecahedron as presented on the right in~\cref{fig:Petersen} ($a$ and $a'$ are the same vertex).
\begin{figure}[h]
\centering
\begin{tikzpicture}
\graph[math nodes, clockwise]
    { subgraph I_n [V={0,1,2,3,4}] --
      subgraph C_n [V={5,6,7,8,9},radius=1.25cm];
      {[cycle] 0,2,4,1,3} };

\end{tikzpicture}\quad
\begin{tikzpicture}
\graph[math nodes, clockwise]
    { subgraph C_n [V={0,2,4,1,3},radius=1.25cm];
      subgraph C_n [V={5,6,7,8,9,5',6',7',8',9'},radius=2.25cm];
        {[cycle] 0,5'}; {[cycle] 2,7'};
      {[cycle] 4, 9'};  {[cycle] 1,6};  {[cycle] 8,3};
 };
\end{tikzpicture}
    \caption{The Petersen graph $P_{5,2}$ and the hemi-dodecahedron}
    \label{fig:Petersen}
\end{figure}
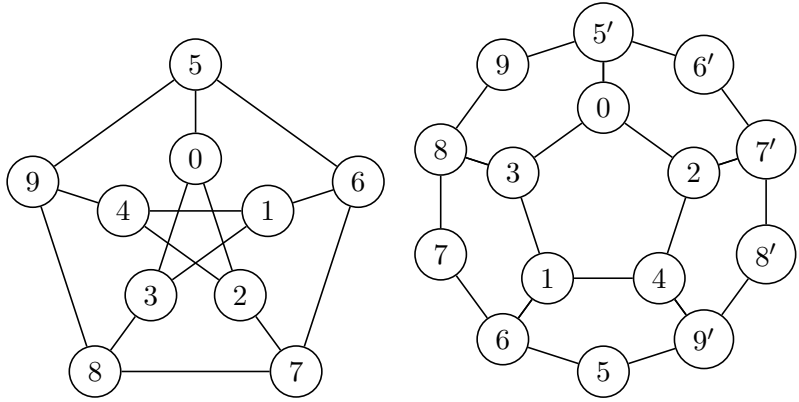

The $(6,2)$-cycle-cover forms a (vertex-transitive) $5$-uniform \v Solt\'es' hypergraph of order $10$ and size $6$, with $V(H)=\{0,1,2,\ldots,9\}$ and $$E(H)=\{(0, 1, 2, 3, 4), (0, 2, 5, 6, 7), (1,3,6,7,8), (2, 4, 7, 8,9), ( 0,3,5,8,9), (1, 4, 5,6, 9)\}.$$
Indeed, it is easy to verify that $W(H)=45$ and also $W(H \setminus v)=45$ for every $v \in [10].$ The latter since the hypergraph is vertex-transitive and taking e.g. $v=5$, we note that $\diam(H \setminus v)=2$ and the pairs of vertices which are at distance $2$ from each other contain at least one neighbour of $v.$
This results in $W(H \setminus v)=\binom{9}{2}+3\cdot 4 -\binom{3}{2}$, since each neighbour of $v$ belongs to one $C_5$ and thus has $4$ neighbours and $4$ non-neighbours within $H \setminus v.$
\end{exam}

Finally, we determine the \v Solt\'es' hypergraph with minimum size. It is isomorphic to a $(10,6,3)$-design (BIBD). A structure with $10$ elements, $5$ edges containing $6$ elements each, such that every element belongs to $3$ edges and every $2$ elements share precisely $3$ edges.
See~\cref{fig:BIBD}. Here the elements/ vertices are the rows, and an edge is presented with a colour containing a cell of the $6$ corresponding rows.

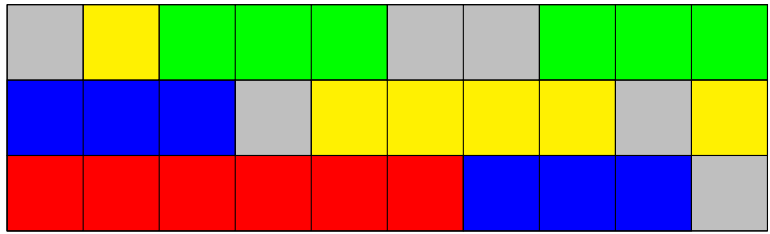
\begin{figure}[h]
    \centering
    \begin{tikzpicture}
        \filldraw [fill=lightgray] (0,0) rectangle (10,3);
      \filldraw [fill=red] (0,0) rectangle (6,1);
      \filldraw [fill=blue] (0,1) rectangle (3,2);
      \filldraw [fill=blue] (6,0) rectangle (9,1);
      \filldraw [fill=green] (2,2) rectangle (5,3);
      \filldraw [fill=green] (7,2) rectangle (10,3);
      \filldraw [fill=yellow] (2,2) rectangle (1,3);
      \filldraw [fill=yellow] (8,2) rectangle (4,1);
      \filldraw [fill=yellow] (9,2) rectangle (10,1);
  \draw[] (0,0) grid (10,3);

\end{tikzpicture}
    \caption{A $(10,6,3)$-BIBD; a $6$-uniform \v Solt\'es' hypergraph with order $10$ and size $5$ }
    \label{fig:BIBD}
\end{figure}

\begin{thr}
    The dual of $K_5^{(3)}$ is the unique uniform \v Solt\'es' hypergraph of minimum size.
\end{thr}

\begin{proof}
    Note that a \v Solt\'es' hypergraph has minimum degree at least $2$, since if not, a vertex $v$ of degree $1$ has a neighbour whose removal disconnects $v$ from the remaining vertices.
    Also the size $m$ and maximum degree satisfy $m \ge \Delta+2,$ as otherwise removing a vertex of maximum degree leaves either a disconnected hypergraph, or a clique (with one hyperedge).
    
    Now $2 \le \delta \le \Delta \le m-2$ implies that $m \ge 4.$
    We first exclude $m=4.$
\begin{claim}
    A \v Solt\'es' hypergraph cannot have size $4$ or less.
\end{claim}

\begin{claimproof}
    In this case, $\Delta=\delta=2$ and by~\cref{thr:smallestSH}, $n\ge 6.$
    No two vertices can belong to the same set of hyperedges (otherwise removing one of them disconnects the other), so $n \le \binom{m}{2}.$ The only case where equality holds, is the $3$-uniform hypergraph on $6$ vertices $K_4^\star$, which is not a \v Solt\'es' hypergraph.
\end{claimproof}

    For a $k$-uniform hypergraph of size $5$ satisfying $\delta=2$ and $\Delta=3$, with vertices $u$ and $v$ having degree $2$ and $3$ respectively,
    we have $W(H \setminus u)< W(H \setminus v)$ if both are finite.
    So we conclude that a $k$-uniform \v Solt\'es' hypergraph of size $5$ is $2$- or $3$-regular.
    Then the order $n$ satisfies $n \Delta = 5k $ by the analogue of the handshaking lemma and thus $5 \mid n.$
    Together with $n\ge 6$ and $n \le \binom{m}{2}=10$, this implies that $n=10.$
    The equality $n= \binom{m}{2}$ implies that we know the hypergraph exactly as a function of $\Delta \in \{2,3\}.$
    We have $5$ edges, and every pair resp. triple of edges intersects in exactly $1$ vertex.
    The $4$-uniform hypergraph ($K_5^\star$) is not a \v Solt\'es' hypergraph, but the $6$-uniform hypergraph ($\left(K_{5}^{(3)}\right)^\star$) is; $W(H)=\binom{10}{2}=\binom{9}{2}+3^2=W(H\setminus v)$.
    As we have ruled out the other possibilities, we conclude that $\left(K_{5}^{(3)}\right)^\star$ is the unique uniform \v Solt\'es' hypergraph of size $5.$
\end{proof}

With a little bit more work, the above theorem can be generalized to arbitrary hypergraphs.

\begin{remark}
        Noting that a \v Solt\'es' hypergraph of size $5$ satisfies $\delta \ge 2, \Delta \le 3$ and the incident edges of a vertex cannot be a subset of the incident edges of another vertex, we know that the candidates correspond with the dual of an antichain in $\binom{[5]}{2,3}$. For $n \ge 9,$ the antichain is uniform and hence there are only $4$ candidates.
        A brute force check (\cite[doc. \text{NonUniformSize5}]{C24})
        confirms that $\left(K_{5}^{(3)}\right)^\star$ is also the unique (not necessarily uniform) \v Solt\'es' hypergraph of size $5.$
\end{remark}

We conclude that a \v Solt\'es' hypergraph has order at least $5$ and size at least $5$, and for both quantities there is a unique \v Solt\'es' hypergraph attaining equality.

Analogous to the question by \v Solt\'es'~\cite{Soltes91}, one can ask to characterize all the uniform \v Solt\'es' hypergraphs.
Given the two infinite families and additional examples, it may be hard to characterize all them. 
This intuition may be even clearer in the case of non-uniform \v Solt\'es' hypergraphs, where we also gave various non-regular examples.
A few potential subquestions one may pose to have an idea on the richness of uniform \v Solt\'es' hypergraphs, are the following ones (the third subquestion might be as hard as finding a \v Solt\'es' graph different from $C_{11})$. 

\begin{q}
\begin{itemize}
    \item There is no uniform \v Solt\'es' hypergraph with order bounded by $9$?
 \item For every $n\ge 10$, there exists an uniform \v Solt\'es' hypergraph with order $n$?\\
 (Note that we have e.g. found examples for every $10 \le n \le 13$)
 \item There exist $3$-uniform \v Solt\'es' hypergraphs?  
 \item There exist $k$-uniform \v Solt\'es' hypergraphs for every $k \ge 4$? \\(Note that we have e.g. found examples for every $4 \le k \le 6$)
 \item There are uniform \v Solt\'es' hypergraphs which are not regular? \\(in \cref{rem:4unif_diam1&nonVT} we mentioned there are non-vertex-transitive examples)
 \end{itemize}
\end{q}

We expect that all of these subquestions may well be answered with yes.
The latter would indicate that the existence of a \v Solt\'es' graph different from $C_{11}$ cannot easily be excluded, given the richness of the class of uniform \v Solt\'es' hypergraphs. For that reason, any negative answer on one of the last three questions would be even more surprising or interesting.

\section*{Acknowledgments}
The author thanks Dragan Stevanovi\'c, Pablo Spiga, Primo{\v{z}} Poto{\v{c}}nik and Marston Conder for referrals, updating on the current censuses, and generating and uploading the full census (of arc-transitive cubic graphs with order bounded by $10^4$).

\bibliographystyle{abbrv}
\bibliography{ref}

\appendix

 \section{Large diameter uniform \v Solt\'es' hypergraphs }\label{sec:appb}

Here we give the graph6-strings of four $4$-regular graphs for which its dual is a \v Solt\'es' hypergraph.

\begin{verbatim}
IYIYMOre_

~??~s`IB?oG?_B?@_A??_GC??o??o?@O?A_??O??A???W???S???c??
@O???G???@?@??C????W????K????H????AOG??AO????B?????CA????
Ca????E??????o?????G_????AO?????@_?????@AG?????W??????
B???????g?O????G_??????B???????BC???????w???????E???????
AO???????B????????E_???????@_???????AOC???????K????????
@g????????B_????????B?????????B_?????????X_????????F_


~?@Ss`AA?_C@?A?G?C?@??H??a?G??G_?C??G??@???C?G?G???CG??@@??GG??
@@???A????C???A?@??@?????O????B?????C?????G?????A@?????OG???OC?????
_G????@??????A?C????A?W?O?E???????_??????C??A????O??O????W???????
G???????C?W?G???@???C????A?G??????A?G??????c???????_C???????C?
G???????O?G?????E?????????G?W?C?????O???O?????O???_?????E??????????
O?C???????@_?????????S??????????@??`???????@O??????????G?CC?????@?
@?????????G?O?????????C?@_?O??????A??_@???????_???W???????O?A?
GC??????A_??g?????????W???_???????D?W???????????GCO??????????K?
D??????????`?@?????????O?OA?????????A?Cc?????????????ACa???????????__c
\end{verbatim}
\begin{verbatim}
~?@os`AQ__C@_D?A?@??OOK?@O?@O?@???_??W??I??@G???o???_???O???C?O?
@_???S???@G????W????I????A?????O????B?????Q?????K?????a?????@_?????
K??????_?????@??????@??O???@_?????@G??????E??????AG???????o???????
o???????S???????C????????_???????E????????c????????S????????
K????????`?????????W????????@_????????B?????????A?????????
@??????????O??O??????E?????????@G?????????@O?????????@_?????????
GO??????????K??????????@_??????????E???????????I???????????
G???????????C???????????B???????????@G???????????A_???????????
D????????????E???????????@@?????????????o????????????E?????????????
W?????????????o?????????????_?????????????O?O???????????C?Q?
O?????????@_?g???????????Q?W????????????S?E????????????
S@O????????????K?S???????????@@???????????????W??????????????
@_??????????????B???????????????B???????????????@O???????????????O?
O?????????????CA????????????????W????????????????o????????????????
o????????????????W????????????????E?????????????????O?
[??????????????A@@@???????????????E?K???????????????E?
o???????????????B@_????????????????p_????????????????Eo?
\end{verbatim}

\section{More examples of non-uniform \v Solt\'es' hypergraphs}\label{sec:app}

\begin{prop}\label{prop:nonunifSH_diam2}
 There exist infinitely many non-uniform \v Solt\'es' hypergraphs with diameter $2$.
\end{prop}

\begin{proof}
 For every $k \ge 2$, we construct a hypergraph $H_k$, obtained by taking a complete $k$-partite graph whose partition classes have size $2k+3$, minus a $2$-factor and adding a hyperedge for every partition class. 
 Examples for $H_2$ and $H_3$ (there are multiple choices for the removed $2$-factor) are depicted in~\cref{fig:H_2&H_3}, with the hyperedges for a partition class drawn with dashed lines, and the removed edges in red. 
 The important observation is that both $H$ and 
 For every $2$ vertices in $H \setminus v,$ which do not both belong to the same hyperedge of size $2k+3$ as $v$ in $H$, the distance within $H \setminus v$ is unchanged. The distance between vertices of the same partition class, changes from $1$ to $2.$
 Now, we observe that $\sum_{u \in V \setminus v} d_{H_k}(u,v)=(n-1)+2=k(2k+3)+1=\binom{2k+2}{2},$ i.e., the sum of distances using $v$ initially equals the number of distances (between the vertices sharing the same hyperedge of size $2k+3$ with $v$ in $H_k$) that increase (by one) in $H_k \setminus v$. 
\end{proof}

\begin{figure}[h]
 \centering
 \begin{tikzpicture}[rotate=90]
 \foreach \x in {0,1,...,6}{
 \foreach \y in {0,1,...,6}{
\draw[black!40!white](\x,0) -- (\y,2.4);
}
}
\foreach \x in {0,1,...,5}{
\draw[red,thick](\x,0) -- (\x+1,2.4);
\draw[red,thick](\x,0) -- (\x,2.4);
}
\draw[red, thick](6,2.4)--(6,0) -- (0,2.4);

\foreach \x in {0,1,...,6}{
\draw[fill] (\x,0) circle (0.15);
}
\foreach \y in {0,1,...,6}{
\draw[fill] (\y,2.4) circle (0.15);
}

\draw[dashed] (3,0) ellipse (3.5cm and 0.75cm);
\draw[dashed] (3,2.4) ellipse (3.5cm and 0.75cm);
 \end{tikzpicture}\quad
 \begin{tikzpicture}[scale=0.65]
 \foreach \x in {0,1,...,8}{
 \foreach \y in {0,1,...,8}{
 \draw[black!20!white](-2+0.5*\x,0.866*4-0.866*\x) -- (6+0.5*\y,-4*0.866+0.866*\y);
 };
 }

 \foreach \z in {0,1,...,8}{
 \foreach \y in {0,1,...,8}{
 \draw[black!20!white](\z,8*0.866) -- (6+0.5*\y,-4*0.866+0.866*\y);
 };
 }

 \foreach \z in {0,1,...,8}{
 \foreach \x in {0,1,...,8}{
 \draw[black!20!white](\z,8*0.866) -- (-2+0.5*\x,0.866*4-0.866*\x);
 };
 }

 \foreach \z in {7,8}{
 \foreach \y in {7,8}{
 \draw[red, thick](\z,8*0.866) -- (6+0.5*\y,-4*0.866+0.866*\y);
 };
 }
 
 \foreach \z in {5,6}{
 \foreach \y in {5,6}{
 \draw[red, thick](\z,8*0.866) -- (6+0.5*\y,-4*0.866+0.866*\y);
 };
 }

 \foreach \z in {2,3}{
 \foreach \x in {2,3}{
 \draw[red, thick](\z,8*0.866) --(-2+0.5*\x,0.866*4-0.866*\x);
 };
 }
 
 \foreach \z in {0,1}{
 \foreach \x in {0,1}{
 \draw[red, thick](\z,8*0.866) -- (-2+0.5*\x,0.866*4-0.866*\x);
 };
 }

\foreach \x in {7,8}{
 \foreach \y in {0,1}{
 \draw[red, thick](-2+0.5*\x,0.866*4-0.866*\x) -- (6+0.5*\y,-4*0.866+0.866*\y);
 };
 }

 \draw[red, thick](0,0) -- (8,0)--(4,8*0.866)--(0,0);
 
 \foreach \x in {5,6}{
 \foreach \y in {2,3}{
 \draw[red, thick](-2+0.5*\x,0.866*4-0.866*\x) -- (6+0.5*\y,-4*0.866+0.866*\y);
 };
 }

 \draw[rotate=120,dashed] (0,0) ellipse (4.5cm and 0.75cm);
 \draw[rotate=0,dashed] (4,8*0.866) ellipse (4.5cm and 0.75cm);
 \draw[rotate=-120,dashed] (-4,8*0.866) ellipse (4.5cm and 0.75cm);

 \foreach \x in {0,1,...,8}{
 \draw[fill] (-2+0.5*\x,0.866*4-0.866*\x) circle (0.15);
 }
 \foreach \y in {0,1,...,8}{
 \draw[fill] (6+0.5*\y,-4*0.866+0.866*\y) circle (0.15);
 }
 \foreach \z in {0,1,...,8}{
 \draw[fill] (\z,8*0.866) circle (0.15);
 }

 \node at (2,7.26) {\large $u_3$};
 \node at (3,7.26) {\large $u_4$};

 \end{tikzpicture}
 
 \caption{Two \v Solt\'es' hypergraphs with diameter $2$.}
 \label{fig:H_2&H_3}
\end{figure}
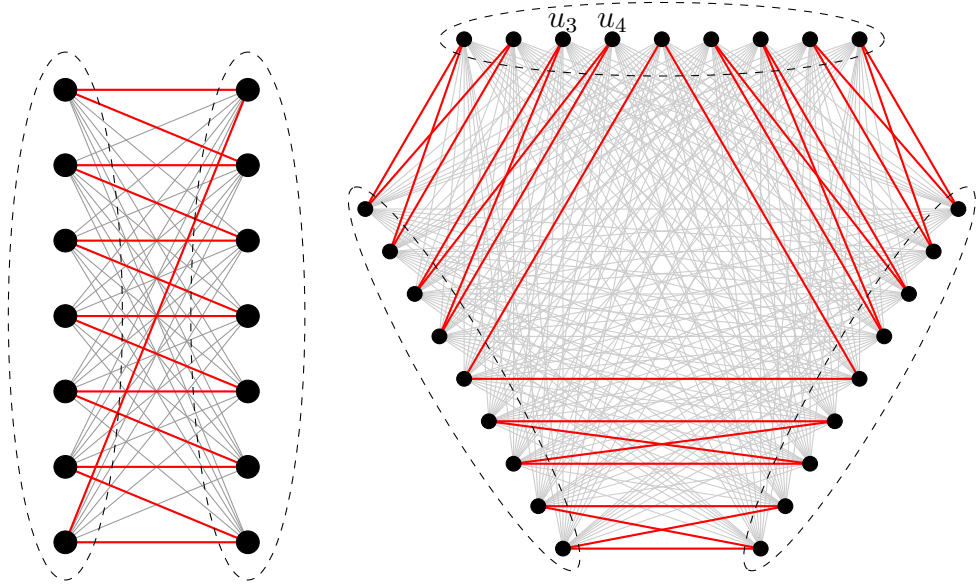

\begin{exam}
 Let $k$ be even and sufficiently large.
 First we take a graph $G$ with vertex set $A \cup B$.
 Here $B=B_1 \cup B_2$ is a set with $k+2$ vertices, where the $k-8$ vertices from $B_1$ have $k-4$ neighbours in $B$ and the $10$ vertices in $B_2$ have $k-3$ neighbours in $B$,
 i.e., $\deg_{G[B]}(v)=k-4$ if $v \in B_1$ and $\deg_{G[B]}(v)=k-3$ if $v \in B_2$.
 Let $A=A_1 \cup A_2$ with $\abs{A_1}=k-2$ and $\abs{A_2}=2,$ such that $\deg_{G[A]}(v)=k-6$ if $v \in A_1$ and $\deg_{G[A]}(v)=k-7$ if $v \in A_2$.
\end{exam}

 In the bipartite graph $G[A,B]$, the degree of vertices of $A_1$ have degree $\frac k2+7,$ the ones of $A_2$ $\frac k2+8$, the vertices in $B_1$ degree $\frac k2+6$ and the ones in $B_2$ degree $\frac k2+5.$


 Finally, we add hyperedges $A$ and $B$, to form a hypergraph $H$ with order $n=2k+2.$

 The degree of the vertices in $A$ is equal to $(k-6)+\frac k2+7+1=\frac{3k}{2}+2$
 and the vertices in $B$ have degree $k-4+\frac k2+6+1=\frac{3k}{2}+3.$
 Hence the graph is not regular (and thus not vertex-transitive),
 The minimum degree of $H$ equals $\frac34 n+\frac 12$ and as a corollary, $\diam(H)=2$ is also immediate.

 Finally, we verify that $H$ is a \v Solt\'es' hypergraph.
 
 For $v \in A_1$, $\sum_{u \in V \setminus v} d_H(u,v)=2k+1+\frac k2-5=\frac{5k}{2}-4.$
 The only distances that increase (by one)
 correspond with non-edges in $A$ which do not contain $v$, there are $\frac{5k+2}{2}-5=\frac{5k}{2}-4.$
 This implies that $W(H)=W(H \setminus v).$
 The same holds true for $v \in A_2, B_1$ and $B_2$ respectively.

\end{document}